\documentclass{amsart}
\usepackage{ytableau}
\usepackage{tikz-cd}
\usepackage{tikz}
\usepackage{amssymb,amsmath,amsthm}
\usepackage{mathtools}
\newcommand{\al}{\alpha}

\newcommand{\la}{\lambda}
\newcommand{\La}{\Lambda}

\DeclareMathOperator{\taa}{Tab_{\al}(\mu)}
\DeclareMathOperator{\st}{SST_{\la}(\mu)}
\DeclareMathOperator{\sta}{SST_{\al}(\mu)}

\DeclareMathOperator{\STa}{SST_{\al^{+}}(\mu^{+})}
\DeclareMathOperator{\Hom}{Hom}
\DeclareMathOperator{\Ext}{Ext}
\DeclareMathOperator{\D}{\mathrm{\Delta}}

\DeclareMathOperator{\rad}{rad}
\newtheorem{theorem}{Theorem}[section]
\newtheorem{lemma}[theorem]{Lemma}
\newtheorem{proposition}[theorem]{Proposition}
\newtheorem{corollary}[theorem]{Corollary}
\newtheorem{definition}[theorem]{Definition}
\theoremstyle{remark}
\newtheorem{remark}[theorem]{Remark}
\newtheorem{remarks}[theorem]{Remarks}
\newtheorem{example}[theorem]{Example}

\numberwithin{equation}{section}

\title{A periodicity theorem for extensions of Weyl modules}

\author{Mihalis Maliakas}
\address{Department of Mathematics, University of Athens, Athens, Greece}
\curraddr{}
\email{mmaliak@math.uoa.gr}
\thanks{}

\author{Dimitra-Dionysia Stergiopoulou}
\address{Department of Mathematics, University of Athens, Athens, Greece}
\curraddr{School of Applied Mathematical and Physical Sciences, National Technical University of Athens, Athens, Greece}
\email{dstergiop@math.uoa.gr}
\thanks{}



\date{May 9, 2024}

\begin{document}

\begin{abstract}
In this paper, we study periodicity phenomena for modular extensions between Weyl modules and between  Weyl and simple modules of the general linear group that are associated to adding a power of the characteristic to the first parts of the involved partitions.\end{abstract}

\maketitle

\section{Introduction}
This article deals with homological aspects of polynomial representations of the general linear group $G=GL_n(K)$ over an infinite field $K$ of characteristic $p>0$.

For a positive integer $r$, let $S_K(n,r)$ denote the corresponding Schur algebra of $G$. The category of finite dimensional $S_K(n,r)$-modules is equivalent to the category of homogeneous polynomial representations of $G$ of degree $r$. Several important $S_K(n,r)$-modules are indexed by partitions $\la$ of $r$ with at most $n$ parts, such as the Weyl modules $\D(\la)$ and the simple modules $L(\la)$. The study of extension groups between such modules is one of the main topics in the polynomial representation theory of $G$. There are relatively few general results, especially those that relate extension groups that correspond to different degrees $r$. One such result is the row (or column) removal theorem for homomorphism spaces and extension groups or, more generally, the horizontal (or vertical) cuts theorem, see \cite{BG}, \cite{FL}, \cite{Dotil}, \cite{Ku2}. Another such result concerns complements of partitions \cite{FHKD}. One also has the degree reduction theorem \cite{Ku1}. These results are characteristic free. In \cite{He} and \cite{HK}, among other results, equalities between $p$-Kostka numbers are obtained that are related to adding a power of $p$ to the first parts of the involved partitions.

If $\la=(\la_1,\dots,\la_n)$ is a partition and $d$ a positive integer, we denote by $\la^+$ the partition $(\la_1+p^d, \la_2,\dots,\la_n)$. Let $S=S_K(n,r)$ and $S'=S_K(n,r+p^d)$. The main result of the present paper is the following.
\begin{theorem}
	Let $K$ be an infinite field of characteristic $p>0$, let $\la=(\la_1,\dots,\la_n)$, $\mu=(\mu_1,\dots,\mu_n)$  be partitions of $r$ and let $d$ be an integer such that $p^d >r-\la_1$. \begin{enumerate} \item If $\mu_2\le \la_1$, then $\Ext^i_{S}(\D(\lambda), \D(\mu)) \simeq\Ext^i_{S'}(\D(\lambda ^+), \D(\mu ^+))$ for all i. \item If $\la_1 \ge r/2$, then $\Ext^i_{S}(\D(\lambda), L(\mu)) \simeq \Ext^i_{S'}(\D(\lambda ^+), L(\mu ^+))$ for all i.\end{enumerate}
\end{theorem}

Here the $\Ext$ groups are taken in the category of modules over the appropriate Schur algebra.  By a result of Donkin \cite{DoSchur1} this is the same as considering the $\Ext$ groups in the category of rational $G$-modules. 

Let us consider two special cases of (1) and their immediate implications for representations of the symmetric group $\mathfrak{S}_n$. Suppose $p>2$ and $r=n$. First, consider $i=0$. By applying the Schur functor to part (1), we obtain an isomorphism between Hom spaces of Specht modules \[\Hom_{\mathfrak{S}_n}(S^{\mu}, S^{\la})\simeq \Hom_{\mathfrak{S}_{n'}}(S^{\mu^+}, S^{\la^+}),\] where $n'=n+p^d$ (Corollary 5.1). This provides an answer to a question posed by Hemmer \cite[Problem 5.4]{Hem2}. In fact, the previous isomorphism can be generalized to higher extension groups $\Ext^i_{\mathfrak{S}_n}(S^{\mu}, S^{\la})$ by a result of Kleshchev and Nakano  \cite[Theorem 6.4(b)(ii)]{KN}, if $p>3$ and $0 \le i \le p-2$. 

Second, suppose $\mu$ consists of one part, $\mu=(n)$. Then by part (1) of our theorem we have $\Ext^i_{S}(\D(\lambda), \D(n)) \simeq\Ext^i_{S'}(\D(\lambda ^+), \D(n+p^d))$ for all $i$, $\lambda$ and $p$ if $d$ satisfies $p^d>n-\lambda_1$. By \cite[Corollary 6.3(b)(iii)]{KN}, we obtain the following periodicity result in the cohomology of the symmetric group \[H^i(\mathfrak{S}_n, S^{\la }) \simeq H^i(\mathfrak{S}_{n'}, S^{\la^+ }), \] for $1 \le i \le 2p-4$ (Corollary 5.3). This provides an answer to the special case of Problem 8.3.1 of Hemmer \cite{Hem1} that corresponds to $c=(1)$. A related but different result has been obtained by Nagpal and Snowden \cite[Theorem 5.1]{NS} (see Remark 5.4(2)) using the theory of FI-modules. Corollary 5.3 generalizes a result of James valid for $i=0$ \cite[Theorem 24.4]{Jam} and a result of Hemmer valid for $i=1$  \cite[Theorem 7.1.8]{Hem1} that was shown under the slightly stronger assumption $p^d>r$ using different methods.

The main tool in the proof of Theorem 1.1 are the characteristic free projective resolutions $B_{*}(\la) \to \Delta(\la) \to 0$ obtained by Santana and Yudin \cite{SY}. We show that, under the hypotheses of Theorem 1.1, the complexes $\Hom_S(B_{*}(\la), \D(\mu))$ and $\Hom_{S'}(B_{*}(\la^+), \D(\mu^+))$ are in fact isomorphic. For this we use computations with tableaux, properties of multinomial coefficients $\mod p$ and combinatorics of Schur algebras associated to the correspondence $\xi_\omega \mapsto \xi_{\omega^+}$ of basis elements,  where $\omega^+$ is obtained from the $n \times n$ matrix $\omega$ by adding $p^d$ to the entry in position $(1,1)$.

There are some very specific cases where projective resolutions of Weyl modules or initial segments thereof are known that are different from those of \cite{SY}. At least in two cases, these have fewer projective summands in each degree. Using these in place of \cite{SY}, we observe that the bound $p^d >r-\la_1$ in the statement of Theorem 1.1 may be improved for (a) $\Hom_{S}(\D(\la), \D(\mu))$, where $\la, \mu$ are as in Theorem 1.1 and (b) $\Ext_S^i(\D(h), \D(\mu)),$ where $h$ is a hook partition (see Section 6). 

The paper is organized as follows. In Section 2 we establish notation and recall basic facts on Weyl modules. In Section 3 we prove a crucial lemma on isomorphisms of weight subspaces of Weyl modules. Section 4 contains the combinatorial results needed for the proof of the main theorem which is given in subsections 4.4 and 4.6. In Section 5 we discuss two consequences of the main result for representations of the symmetric group. In Section 6 we observe that the bound $p^d>r-\la_1$ may be improved in two specific instances.

\section{Notation and recollections}
\subsection{Notation}

Throughout this paper, $K$ will be an infinite field of characteristic $p>0$. We fix positive integers $n$ and $r$. We will work mostly with homogeneous polynomial representations of $GL_n(K)$ of degree $r$, or equivalently, with finite dimensional modules over the Schur algebra $S=S_K(n,r)$. A standard reference here is \cite{Gr}.

Let $V=K^n$ be the natural $GL_n(K)$-module. The divided power algebra $DV=\sum_{i\geq 0}D_iV$ of $V$ is defined as the graded dual of the Hopf algebra $S(V^{*})$, where $V^{*}$ is the linear dual of $V$ and $S(V^{*})$ is the symmetric algebra of $V^{*}$, see \cite[I.4]{ABW}. 

We denote by $\Lambda(n;r)$ the set of sequences $\al=(\al_1, \dots, \al_n)$ of length $n$ of nonnegative integers that sum to $r$ and by $\Lambda^+(n,r)$ the subset of $\Lambda(n,r)$ consisting of sequences $\lambda=(\lambda_1, \dots, \lambda_n)$ such that $\lambda_1 \ge \lambda_2 \ge \dots \ge \lambda_n$. Elements of $\Lambda^+(n,r)$ are referred to as partitions of $r$ with at most $n$ parts. 


If $\al=(\al_1,\dots, \al_n) \in \Lambda(n,r)$, we denote by $D(\al)$ or $D(\al_1,\dots,\al_n)$ the tensor product $D_{\al_1}V\otimes \dots \otimes D_{\al_n}V$. All tensor products in this paper are over $K$.

For $\lambda \in \Lambda^+(n,r)$, we denote by $\Delta(\lambda)$ the corresponding Weyl module for $S$ and by ${\pi}_{\D(\la)} : D(\la) \to \D(\la)$ the natural projection.

Let $\La(n,n;r)$ and $\La(n,n,n;r)$ be the sets of functions $(\omega_{ij}):\mathbf{n} \times \mathbf{n} \to \mathbb{N}$ and $(\theta_{stq}):\mathbf{n} \times \mathbf{n} \times \mathbf{n} \to \mathbb{N}$ respectively, where $\mathbf{n}=\{1,2,\dots,n\}$ and $\mathbb{N} = \{0,1,2,\dots\},$ such that $\sum_{i,j=1}^{n}\omega_{ij}=r$ and $\sum_{s,q,t=1}^{n}\theta_{sqt}=r$. If $d$ is a positive integer and $\al=(\al_1,\dots,\al_n) \in \La(n;r),$ $\omega =(\omega_{st}) \in \La(n,n;r)$ and $\theta =(\theta_{stq}) \in \La(n,n,n;r)$, define $\al^+ \in \La(n;r')$, $\omega^+ \in \La(n,n;r')$ and $\theta^+ \in \La(n,n,n;r)$ by \begin{align*}&\al^+=(\al_1+p^d,\al_2,\dots,\al_n), \\ &\omega^+_{st}=\begin{cases}\omega_{11}+p^d& (s,t)=(1,1)\\\omega_{st}& (s,t)\neq (1,1),\end{cases} \\ &\theta^+_{stq}=\begin{cases}\theta_{111}+p^d& (s,t,q)=(1,1,1)\\\theta_{stq}& (s,t,q)\neq (1,1,1).\end{cases}\end{align*}

Next we recall some results need for the proof of Theorem 1.1.

\subsection{Semistandard basis of $\Delta(\mu)$} We will record here a fundamental fact from \cite{ABW}.
Consider the order $e_1<e_2<\dots<e_n$ on the natural basis $\{e_1,F,e_n\}$ of $V=K^n$. We will denote each element $e_i$ by its subscript $i$. For a partition $\mu=(\mu_1,\dots,\mu_n) \in \Lambda^+(n,r)$, a tableau of shape $\mu$ is a filling of the diagram of $\mu$ with entries from $\{1,\dots,n\}$. The set of tableaux of shape $\mu$ will be denoted by $\mathrm{Tab}(\mu)$. 

A tableau is called \textit{row semistandard} if the entries are weakly increasing across the rows from left to right.  A row semistandard tableau is called \textit{semistandard} if the entries are strictly increasing in the columns from top to bottom. (We should point out that the terminology used in \cite{ABW} is 'co-standard'). The set of semistandard tableaux of shape $\mu$ will be denoted by $\mathrm{SST}(\mu)$.

The \textit{weight} of a tableau $T$ is the tuple $\alpha=(\alpha_1,\dots,\alpha_n)$, where $\alpha_i$ is the number of appearances of the entry $i$ in $T$. The subset of $\mathrm{SST}(\mu)$ consisting of the semistandard tableaux of weight $\alpha$ will be denoted by $\mathrm{SST}_{\alpha}(\mu).$ 

For example, let $n=4$. The following tableau of shape $\mu=(6,4)$ 
\begin{center}
	$T=$
	\begin{ytableau}
		\ 1&1&1&2&2&4\\
		\ 1&2&3&4
	\end{ytableau}
\end{center}
is row semistandard but not semistandard because of the violation in the first column. It has weight $\alpha=(4,3,1,2)$. We will use 'exponential' notation for row semistandard tableaux. Thus for the previous example we write \[ T=\begin{matrix*}[l]
	1^{(3)}2^{(2)}4 \\
	1234\end{matrix*}. \]

To each row semistandard tableau $T$ of shape $\mu=(\mu_1, \dots,\mu_n)$ we may associate an element $x_T=x_T(1) \otimes \cdots \otimes x_T(n) \in D(\mu_1,\dots,\mu_n),$ where $x_T(j)=1^{(a_{1j})}\dots n^{(a_{nj})}$ and $a_{ij}$ is equal to the number of appearances of $i$ in the $j$-th row of $T$. For $T$ in the previous example we have $x_T=1^{(3)}2^{(2)}4\otimes1234$. According to \cite[Theorem II.2.16]{ABW} we have the following. \begin{theorem}[\cite{ABW}] The set $\{\pi_{\D(\mu)}(x_T): T \in \mathrm{SST}(\mu)\}$ is a basis of $\Delta({\mu})$.\end{theorem}

If $U \in \mathrm{Tab}(\mu)$, we will denote the element $\pi_{\D(\mu)}(x_U) \in \Delta(\mu)$ by $[U].$ According to the previous theorem, there exist $c_T \in K$ such that \begin{equation}[U]=\sum_{T \in \mathrm{SST}(\mu)}c_T[T].\end{equation} From the proof of the previous theorem in \cite{ABW} (in particular, last line of p. 235 and the top of p. 236), it follows that each semistandard tableau $T$ appearing in the right hand side of eqn. (2.1) with nonzero coefficient is obtained by a rearrangement of the entries of $U$. Thus we have the following remark. \begin{remark}Each semistandard tableau $T$ appearing in the right hand side of eqn. (2.1) with nonzero coefficient has weight equal to the weight of $U$.\end{remark}

\subsection{Weight subspaces of $\Delta(\mu)$} Let $\al, \beta  \in \Lambda(n,r)$ and let $\omega =(\omega_{ij}) \in \La(n,n;r)$ such that $\sum_{j=1}^{n}{\omega_{ij}}=\al_i $ (for each $i$) and $\sum_{i=1}^{n}{\omega_{ij}}=\beta_j $ (for each $j$), where $\al= (\al_1,\dots, \al_n)$ and $\beta= (\beta_1,\dots, \beta_n)$. For each $i=1,\dots,n$ consider the indicated component \begin{align*}\Delta: D(\al_i) \to D(\omega_{i1},\omega_{i2},\dots,\omega_{in}),\end{align*} of the comultiplication map of the Hopf algebra $DV$. If $x \in D(\alpha_i)$, the image $\D(x) \in D(\omega_{i1},\omega_{i2},\dots,\omega_{in})$ is a sum of elements of the form $x_s(\omega_{i1},1) \otimes x_s(\omega_{i2},2) \otimes \cdots \otimes x_s(\omega_{in},n)$, where for each $i$ we have $x_s(\omega_{ij},j) \in D(\omega_{ij})$. By a slight abuse of notation we will write $x_s(\omega_{ij})$ in place of $x_s(\omega_{ij},j)$. Thus we will write \[\Delta(x) = \sum_{s}x_s(\omega_{i1}) \otimes x_s(\omega_{i2}) \otimes \cdots \otimes x_s(\omega_{in}).\]
\begin{definition}
	With the previous notation, define the map $ \phi_{\omega}:D(\al) \to D(\beta)$ that sends  $x_1 \otimes x_2 \otimes \dots\otimes x_n$ to \[\sum_{s_1,\dots,s_n}x_{1s_1}(\omega_{11}) \cdots x_{ns_n}(\omega_{n1}) \otimes  \cdots \otimes x_{1s_1}(\omega_{1n})\cdots x_{ns_n}(\omega_{nn}). \]
\end{definition}

Now suppose that $\beta = \mu$. For a row standard $T \in \taa$, \[T=\begin{matrix*}[l]
	1^{(a_{11})}2^{(a_{21})}\cdots n^{(a_{n1})}\\
	\cdots\\
	1^{(a_{1n})}2^{(a_{2n})}\cdots n^{(a_{nn})},\end{matrix*}\]
define $\omega(T)=(\al_{ij})$. The corresponding map $ \phi_{\omega(T)}:D(\al) \to D(\mu)$ will be denoted by $\phi_T$ for short.

Let us denote by $\D(\mu)_{\al}$ the weight subspace of the Weyl module $\Delta(\mu)$ corresponding to the weight $\al$. According to \cite[Section 2]{AB}, we have the following result.
\begin{proposition}[\cite{AB}] Let $\al \in \Lambda(n,r)$ and $\mu \in \Lambda^+(n,r)$. Then there is an isomorphism of vector spaces $\D(\mu)_{\al} \simeq \Hom_S(D(\al), \Delta(\mu))$ such that $[T] \mapsto \pi_{\D{(\mu)}} \circ \phi_T$ for all row semistandard $T \in \mathrm{Tab}_{\al}(\mu)$. Moreover, a basis of $\Hom_S(D(\al), \Delta(\mu))$ is the set $\{\pi_{\D{(\mu)}} \circ \phi_T: T \in \mathrm{SST}_{\al}(\mu)\}.$
\end{proposition}

\section{Isomorphism of weight spaces}
If $T \in \mathrm{Tab}_{\al}(\mu)$, let $T^+ \in \mathrm{Tab}_{\al^+}(\mu^+)$ be obtained from $T$ by inserting $p^d$ 1's in the beginning of the top row.

From Proposition 2.4 we have the basis elements $\pi_{\D{(\mu)}} \circ \phi_T$ of the vector space $ \Hom_S(D(\al),\Delta(\mu))$, where $T \in \sta$ and likewise we have the basis elements $\pi_{\D(\mu^+)} \circ \phi_{X}$ of $ \Hom_{S'}(D(\al^+),\Delta(\mu^+))$, where $X \in \STa.$
\begin{lemma}Let $\al \in \Lambda(n,r)$ and $\mu \in \Lambda^+(n,r)$ such that $\mu_2 \le \al_1$. Then the maps \begin{align*}&\sta \to \STa, T \mapsto T^+\\ &\STa \to \sta, T^+ \mapsto T\end{align*}
	are inverses of each other and bijections. Hence the map  \begin{align*}\Hom_S(D(\al), \Delta(\mu)) &\to \Hom_{S'}(D(\al^+), \Delta(\mu^+)),\\ \sum_{T \in \sta}c_T\pi_{\D(\mu)} \circ \phi_T &\mapsto \sum_{T \in \sta}c_T \pi_{\D(\mu^+)} \circ\phi_{T^+},\end{align*}
	is an isomorphism.
\end{lemma}
\begin{proof} From the definition of semistandard tableau, it is clear that if $T \in \sta$, then $T^+ \in \STa.$ It is also clear that the first map of the lemma is injective.
	
	If $X \in \STa$, then there are no 1's below the first row. Hence the first row of $X$ contains $\al_1+p^d$ 1's. Let $T \in \mathrm{Tab}_{\al}(\mu)$ be the tableau obtained from $X$ by deleting $p^d$ 1's. Since $\mu_2 \le \al_1$, we have that $T$ is semistandard, i.e. $T \in \sta$. It is clear that $T^{+}=X$. \end{proof}

\subsection{Straightening law for two rows}
First we record from \cite[Lemma 4.2]{MS3}, the following that is a particular case of the straightening law for Weyl modules with two parts.
\begin{lemma}Let $\mu=(\mu_1, \mu_2)$ be a partition consisting of two parts and consider an element $[T]=\begin{bmatrix*}
		1^{(a_1)}   \cdots   n^{(a_n)} \\
		1^{(b_1)}  \cdots  n^{(b_n)}
	\end{bmatrix*} \in \Delta(\mu).$ Then we have the following. \begin{enumerate} \item If $a_1+b_1 > \mu_1$, then $[T]=0$. \item If $a_1+b_1 \le \mu_1$, then \begin{align*}
			[T]=(-1)^{b_1}\sum_{i_2,\dots,i_n} \tbinom{b_2+i_2}{b_2} \cdots \tbinom{b_n+i_n}{b_n}\begin{bmatrix*}[l]
				1^{(a_1+b_1)}  2^{(a_2-i_2)}  \cdots   n^{(a_n-i_n)} \\
				\noindent 2^{(b_2+i_2)} \cdots n^{(b_n+i_n)}
			\end{bmatrix*},\end{align*} where the sum ranges over all nonnegative integers $i_2, \dots, i_n$ such that $i_2+\cdots+i_n=b_1$ and $i_s \le a_s$ for all $s=2,\dots,n$. \end{enumerate}
\end{lemma}
The point of the above lemma is that in case (2), the coefficients that appear do not depend on $a_1$.
\subsection{A crucial lemma} We want to show that the map in the last statement  of Lemma 3.1 may be defined in a similar way without the assumption that the tableaux $T$ are semistandard. Thus we need to show compatibility with the straightening law. We will use ideas from Lemma II.2.15 of \cite{ABW} which is stated and proven for costandard modules but the main steps are valid for Weyl modules also, see Lemma II.3.15 loc.cit.
\begin{lemma} Let $\al \in \Lambda(n,r)$ and $\mu \in \Lambda^+(n,r)$ such that $\mu_2 \le \al_1$. \begin{enumerate}\item Let $U \in \taa$ and write \[
		[U]=\sum_{T \in \mathrm{SST}(\mu)}c_T[T], \;\;
		[U^+]=\sum_{T \in \mathrm{SST}(\mu)}c_{T^+}[T^+], 
		\] in $\Delta(\mu)$ and $\Delta(\mu^+)$ respectively, where $c_T, c_{T^+} \in K$. Then $c_T=c_{T^+}$ for all $T \in \mathrm{SST}(\mu)$.
		\item The map \[\Hom_S(D(\al), \Delta(\mu)) \to \Hom_{S'}(D(\al^+), \; \Delta(\mu^+)), \pi_{\D(\mu)} \circ \phi_U \mapsto \pi_{\D(\mu^+)} \circ\phi_{U^+},\]
		where $U \in \taa$ is row standard, is a well defined isomorphism.
	\end{enumerate}
\end{lemma}
\begin{proof} First some notation. If $\nu =(\nu_1,\dots,\nu_n)$ is a partition and $T \in \mathrm{Tab}(\nu)$ is a tableau, let $\overline{\nu}=(\nu_2,\dots,\nu_n)$  be the partition obtained from $\nu$ by deleting the first part and let $\overline{T}$ be the tableau obtained from $T$ by deleting the first row. Note that $\overline{\nu}=\overline{\nu^{+}}$ and $\overline{T}=\overline{T^{+}}$. 
	
	(1) Let $U \in \taa$. We proceed with a number of reductions.
	
	Step 1. We may assume that the tableau $\overline{U}=\overline{U^+}$ is semistandard. Indeed, By Theorem 2.1 there exist semistandard tableaux $Y \in \mathrm{SST}(\overline{\mu})$ and coefficients $c_{Y} \in K$ such that \begin{equation}\label{str1} [\overline{U}] =[\overline{U^+}]=\sum c_{Y}[Y].\end{equation}  By attaching the first row of $U$ to $\overline{U}$ and each $Y$ and likewise by attaching the first row of  $U^+$ to $\overline{U^+}$ and each $Y$, we obtain (as in the last paragraph of the proof of Lemma II.2.15 of \cite{ABW}) \begin{equation}\label{str2}
		[U]=\sum c_{Y}[Y'] \; \text{and} \; 
		[U^+]=\sum c_{Y}[(Y')^+].
	\end{equation}
	Here we have that the tableaux $\overline{Y'}=\overline{(Y')^+}$ are semistandard.
	
	Step 2. We may assume that the tableau $\overline{U}=\overline{U^+}$ contains at least one 1. Indeed, by Step 1, we may assume that the tableau $\overline{U}=\overline{U^+}$ is semistandard. If it contains no 1, then the tableaux $U$ and $U^+$, after row-standardizing the first row of each, are semistandard because the number of 1's in the first row of $U$ and the first row of $U^+$ is equal to  $\al_1$ and $\al_1 +p^d$ respectively and we have $\al_1 \ge \mu_2$. In this case the conclusion of part (1) of the lemma is clear.
	
	Step 3. We apply Lemma 3.2 to the first two rows of $U$ and to the first two rows of $U^+$ in order to express each as a linear combination of tableaux that have all 1's (if any) in the first row. Since the coefficients appearing in the right hand side of the equation in part (2) of that Lemma do not depend on the number of 1's, we may assume that all the 1's of the tableau $U$ are located in the first row.
	
	Step 4. We repeat Step 1 to the tableau  $\overline{U}=\overline{U^+}$. By the previous step and Remark 2.2, this tableau has no 1's. Thus the $Y$ in the right hand side of eqn. (\ref{str1}) are semistandard and have no 1's. From this and the fact that $\al_1 \ge \mu_2$ we conclude that the $Y'$ and $(Y')^+ $ appearing in the right hand sides of equations (\ref{str2}) are semistandard.

	Since the sets $\{[T] \in \Delta(\mu): T \in \mathrm{SST}(\mu)\}$ and $\{[T^+] \in \Delta(\mu^+): T \in \mathrm{SST}(\mu)\}$  in part (1) of the lemma are linearly independent, the result follows.
	
	(2) This follows from (1) of the lemma and Lemma 3.1.
\end{proof}
An immediate consequence of part (2) of the previous lemma and Proposition 2.4 is the following.
\begin{corollary}\label{cor} Let $\al \in \Lambda(n,r)$ and $\mu \in \Lambda^+(n,r)$ such that $\mu_2 \le \al_1$. Let $c_i \in K$ and $U_i \in \mathrm{Tab}_{\al}(\mu)$, for $i=1,2,\dots,s.$ Then
	$\sum_{i}c_i[U_i] =0 $ in $\D(\mu)$ if and only if $\sum_{i}c_i[U^+] =0 $ in $\D(\mu^+)$.
\end{corollary}

We note that the map of Lemma 3.3(2) may not be well defined if the assumption $\mu_2 \le \al_1$ is relaxed. For example, let $\al=(\mu_1-1,\mu_1+1)$, where $\mu=(\mu_1, \mu_1)$. Then $[U]=0$, but $[U^+] \neq 0$  for all $d>0$, where  $[U]=\begin{bmatrix*}
	1^{(\mu_1-1)}2 \\
	2^{(\mu_1)}  
\end{bmatrix*} $

\section{Periodicity in Ext}
\subsection{Binomial coefficients mod p} If $a, a_1,\dots,a_s$ are nonnegative integers such that $a=a_1+\dots+a_s$, we denote the multinomial coefficient $\frac{a!}{a_1!a_2! \cdots a_s!}$ by $\tbinom{a}{a_1, a_2, \dots, a_s}$. Then \begin{equation}\tbinom{a}{a_1, a_2, \dots, a_s}=\tbinom{a}{a_1}\tbinom{a-a_1}{a_2} \cdots \tbinom{a-a_1 -\cdots -a_{s-1}}{a_s}.\label{multinomial} \end{equation} We will need the following well known property of multinomial coefficients.
\begin{lemma}\label{binom}
	Let $p$ be a prime. \begin{enumerate} \item Let $a,b$ nonnegative integers. If $d$ is a positive integer such that $p^{d}>b$, then $\binom{a+p^{d}}{b} \equiv \binom{a}{b} \mod p.$ \item let $a, a_1,\dots,a_s$ be nonnegative integers such that $a=a_1+\dots+a_s$. If $d$ is a positive integer such that $p^d > a-a_1$, then $\binom{a+p^d}{a_1+p^d,a_2,\dots,a_s} \equiv  \binom{a}{a_1,a_2,\dots,a_s} \mod p.$\end{enumerate}
\end{lemma} \begin{proof} (1) This follows, for example, from Lemma 22.5 of \cite{Jam}.
	
	(2) Using (\ref{multinomial}) and part (1) of the lemma, we have \begin{align*}\tbinom{a+p^d}{a_1+p^d,a_2,\dots,a_s} &= \tbinom{a+p^d}{a_1+p^d} \tbinom{a-a_1}{a_2} \cdots \tbinom{a-a_1-\dots-a_{s-1}}{a_{s}} \\&=\tbinom{a+p^d}{a-a_1} \tbinom{a-a_1}{a_2} \cdots \tbinom{a-a_1-\dots-a_{s-1}}{a_{s}} \\& \equiv \tbinom{a}{a-a_1} \tbinom{a-a_1}{a_2} \cdots\tbinom{a-a_1-\dots-a_{s-1}}{a_{s}} \\& \equiv \tbinom{a}{a_1,a_2,\dots,a_s} \mod p. \end{align*}    \end{proof}

\subsection{The resolutions of Santana and Yudin} We define here the projective resolution of $\D(\la)$ obtained by Santana and Yudin \cite{SY}. Their construction is valid over any commutative ring $R$ with identity but we will continue to  assume here that $R=K$ is an infinite field of characteristic $p>0$. 

First we fix some notation. Recall the dominance partial ordering on $\La(n;r)$. For $\alpha, \beta \in  \La(n;r)$ we write $\alpha \trianglerighteq \beta$ if $\sum_{s=1}^{t}\alpha_s \ge \sum_{s=1}^{t}\beta_s$ for all $t$. If $\alpha \trianglerighteq \beta$ and $\alpha \neq \beta$ we write $\alpha \triangleright \beta$.

If $\omega =(\omega_{st}) \in \La(n,n;r)$ and $\theta =(\theta_{stq}) \in \La(n,n,n;r)$, let $\omega^1, \omega^2 \in \La(n;r)$ and  $\theta^1, \theta^2, \theta^3 \in \La(n,n;r)$ be defined by 
\begin{align*}
	&(\omega^1)_t=\sum_{s=1}^{n}\omega_{st}, \;	(\omega^2)_s=\sum_{t=1}^{n}\omega_{st} \\
	&(\theta^1)_{tq}=\sum_{s=1}^{n}\theta_{stq}, \;	(\theta^2)_{sq}=\sum_{t=1}^{n}\theta_{stq}, \; (\theta^3)_{st}=\sum_{q=1}^{n}\theta_{stq}.
\end{align*}
For $i,j \in I(n,r)$, let us denote the element $\xi_{i,j}$ of the Schur algebra $S=S_K(n,r)$ defined in 
\cite[Section 2.3]{Gr} by $\xi_{\omega}$, where $\omega =(\omega_{st}) \in \La(n,n;r)$ is defined by \[\omega_{st} = |\{q \in \{1,2,\dots,r\}: (i_q,j_q)=(s,t)\}|.\]  Also, if $\nu=(\nu_1,\dots,\nu_n) \in \La(n,r)$, we denote by $\xi_\nu$ the element $\xi_\omega$, where $\omega =diag(\nu_1,\dots,\nu_n)$ is the indicated diagonal matrix. 

Second we recall the multiplication rule of Santana and Yudin \cite[Proposition 2.3]{SY}. If $\omega, \pi \in \La(n,n;r)$, then \begin{equation}
	\xi_\omega \xi_\pi = \sum_{\substack{\theta \in \La(n,n,n;r) \\ \theta^3=\omega, \theta^1=\pi}}[\theta]\xi_{\theta^2},
\end{equation}  
where $[\theta]$ is the following product of multinomial coefficients \begin{equation} [\theta]= \prod_{s,t=1}^{n}\tbinom{(\theta^2)_{st}}{\theta_{s1t},\theta_{s2t},\dots,\theta_{snt}}.\end{equation}
In particular, if $\al \in \La(n,r)$, $\omega \in \La(n,n;r)$ and $\al=\omega^2$, then \begin{equation}
	\xi_{\al}\xi{_\omega}=\xi_\omega.
\end{equation}

Let $\La^1(n,n;r)$ be the subset of $\La(n,n;r)$ consisting of the $\omega$ that are upper triangular and not diagonal and let $J$ be the subspace of $S$ spanned by the $\xi_{\omega}$, where $\omega \in \La^1(n,n;r)$. The projective resolution $B_{*}(\la) \to \Delta(\la) \to 0$ of \cite{SY}, denoted there by $B_{*}(W_\la)$, is defined as follows. For $k=0$ we have $B_0(\la)=S\xi_{\la}$ and for $k \ge 1$ we have
\[B_k(\la)= \bigoplus_{\substack{\mu^{(1)}\triangleright \cdots \triangleright\mu^{(k)} \triangleright \la \\ \mu^{(1)},\dots,\mu^{(k)} \in \La(n;r)}} S\xi_{\mu^{(1)}} \otimes \xi_{\mu^{(1)}} J\xi_{\mu^{(2)}}\otimes \cdots \otimes \xi_{\mu^{(k)}}J\xi_{\la}.\]
\begin{definition}\cite[p. 2587]{SY}. If $k>0$ and $\al \triangleright \la$, define $\Omega_k(\la, \al)$ as the set of all $k$-tuples $(\omega_1,\dots,\omega_k)$, where $\omega_1, \dots , \omega_k \in \Lambda^1(n,n;r)$ satisfy \begin{equation}\label{4.1}
		\al=(\omega_1)^2, (\omega_1)^1=(\omega_2)^2, (\omega_2)^1=(\omega_3)^2, \dots , (\omega_{k-1})^1=(\omega_k)^2, (\omega_k)^1=\la.
\end{equation}\end{definition}
A $K$-basis of $B_0(\la)$ is $\{\xi_\omega:\omega \in \La(n,n;r), \omega^1=\la\} $ and a $K$-basis of $B_k(\la)$ for $k\ge 1$ is \[\{\xi_{\omega_0} \otimes\xi_{\omega_1}\otimes \cdots \otimes \xi_{\omega_k}:(\xi_{\omega_0}, \xi_{\omega_1}, \dots, \xi_{\omega_k}) \in \tilde{b_k}(\la)\},\]
where \[\tilde{b_k}(\la)=\{(\omega_0, \omega_1, \cdots, \omega_k): \omega_0 \in \La(n,n;r), (\omega_1,\dots,\omega_k) \in \Omega_k(\la, (\omega_0)^1).\] In terms of these bases, the differential $\partial_k: B_k(\la) \to B_{k-1}(\la)$ of $B_*(\la)$ is given by \begin{equation} 
	\partial_k(\xi_{\omega_0} \otimes \cdots \otimes \xi_{\omega_k})=\sum_{t=0}^{k-1}(-1)^t\xi_{\omega_0} \otimes \cdots \otimes \xi_{\omega_t}\xi_{\omega_{t+1}} \otimes \cdots \otimes \xi_{\omega_k}.
\end{equation}

Moreover, each summand $S\xi_{\mu^{(1)}} \otimes \xi_{\mu^{(1)}} J\xi_{\mu^{(2)}}\otimes \cdots \otimes \xi_{\mu^{(k)}}J\xi_{\la}$ in the right hand side of the definition of $B_k(\la)$, where $k>0$, is isomorphic as an $S$-module to a direct sum of copies of $S\xi_{\mu^{(1)}}$ and the multiplicity of  $S\xi_{\mu^{(1)}}$ is equal to $|\Omega_k(\la, \mu^{(1)})|$.

It will be convenient for our purposes to use the following notation.

\begin{definition} Suppose $k>0$, $\al \triangleright \la$ and $(\omega_1,\dots,\omega_k) \in \Omega_k(\la, \al)$ . \begin{enumerate}\item Let $M[\omega_1,\dots,\omega_k]$$=S\xi_{\al}.$  \item If $k>1$ and $i\ge1$, let \[M[\omega_1,\dots,\omega_i\omega_{i+1},\dots,\omega_k]= \bigoplus_{\theta}M[\omega_1,\dots,\omega_{i-1},\theta^2,\omega_{i+2},\dots,\omega_k],\] where the sum is over all $\theta \in \La(n,n,n;r)$ such that  $\theta^3=\omega_i$ and $\theta^1=\omega_{i+1}.$ \end{enumerate}
\end{definition}
\begin{remark}With the above notation, the restriction of the differential $\partial_k$, where $k>1$, to the summand $M[\omega_1,\dots,\omega_k]$ is the map \[M[\omega_1,\dots,\omega_k]\to M[\omega_2,\dots,\omega_k] \bigoplus \bigoplus_{i} M[\omega_1,\dots,\omega_i\omega_{i+1},\dots,\omega_k]\] whose component $M[\omega_1,\dots,\omega_k]\to M[\omega_2,\dots,\omega_k]$ is right multiplication by $\xi_{\omega_1}$ and whose component $M[\omega_1,\dots,\omega_k]\to M[\omega_1,\dots,\omega_{i-1},\theta^2,\omega_{i+2},\dots,\omega_k]$ is multiplication by $(-1)^{i}[\theta^2] \in K$. For $k=1$ the restriction of the differential $\partial_1$ to $M[\omega_1]$ is the map $M[\omega_1] \to S\xi_{\la}$ that is right multiplication by $\xi_{\omega_1}$.\end{remark}
\subsection{The correspondence $\xi_\omega \mapsto \xi_{\omega^+}$}
The remark and lemma that follow show that for all $k$ there is a bijection between the summands of $B_k(\la)$ and the summands of $B_k(\la^+)$ given by $S\xi_\al \leftrightarrow S'\xi_{\al ^+}$.
\begin{remark}
	It is clear that the following map is a bijection \[\{\al \in \La(n,r): \al \trianglerighteq \la\} \to \{\beta \in \La(n,r+p^d): \beta \trianglerighteq \la^+\}, \al \mapsto \al^+.\]
\end{remark}

\begin{lemma} If $k>0$ and $\al \triangleright \la$, then	$|\Omega_k(\la, \al)| = |\Omega_k(\la^+, \al^+)|.$
\end{lemma}
\begin{proof}If $\omega=(\omega_{ij}) \in \Lambda^1(n,n;r)$ then $\omega^+ \in \Lambda^1(n,n;r')$. Using relations (\ref{4.1}), it easily follows that if $(\omega_1,\dots, \omega_k) \in\Omega_k(\la, \al)$, then  $(\omega_1^+,\dots, \omega_k^+) \in \Omega_k(\la^+, \al^+)$. We claim that the following correspondence is 1-1 and onto
	\[\Omega_k(\la, \al) \ni (\omega_1,\dots, \omega_k) \to (\omega_1^+,\dots, \omega_k^+) \in \Omega_k(\la^+, \al^+).\] 
	
	Indeed, if $(\zeta_1, \dots , \zeta_k) \in \Lambda^1(n,n;r')$, then \begin{equation}
		\al^+=(\zeta_1)^2, (\zeta_1)^1=(\zeta_2)^2, (\zeta_2)^1=(\zeta_3)^2, \dots , (\zeta_{k-1})^1=(\zeta_k)^2, (\zeta_k)^1=\la^+.
	\end{equation} From this and the fact that each $\zeta_i$ is upper triangular with nonnegative entries, we obtain \[\al_1+p^d \ge (\zeta_1)_{11} \ge (\zeta_2)_{11} \ge \cdots \ge (\zeta_k)_{11}=\la_1+p^d,\] where $(\zeta_i)_{st}$ is the entry of $\zeta_i$ at position $(s,t)$. For each $i=1,\dots,k$, define $\omega_i \in \Lambda^1(n,n;r)$ by \[(\omega_i)_{st}=\begin{cases}
		(\zeta_i)_{11}-p^d, &(s,t)=(1,1)\\(\zeta_i)_{st}, &(s,t) \neq (1,1). 
	\end{cases} \] Then one checks that relations (\ref{4.1}) hold, that is $(\omega_1,\dots,\omega_k) \in \Omega_k(\la, \al)$, and $(\omega_1^+,\dots,\omega_k^+)=(\zeta_1,\dots,\zeta_k)$.
\end{proof}

\begin{lemma}
	Let $\omega, \pi \in \La(n,n;r)$. If $\omega$ is upper triangular or if $\pi$ is lower triangular, then the  map \[\{\theta \in \La(n,n,n;r):\theta^1=\pi, \theta^3=\omega\} \to \{\eta \in \La(n,n,n;r'):\eta^1=\pi^+, \eta^3=\omega^+\},\]\[ \theta \mapsto \theta^+, \]
	is a bijection
	
\end{lemma}
\begin{proof}First assume $\omega$ is upper triangular. It is clear from the definitions that if $\theta \in \La(n,n,n;r')$ satisfies $\theta^1=\pi, \theta^3=\omega$, then $\theta^+ \in \La(n,n,n;r)$ and $(\theta^+)^1=\pi^+, (\theta^{+})^3=\omega^+ $.
	
	Suppose $\eta \in \La(n,n,n;r')$ satisfies $\eta^1=\pi^+, \eta^3=\omega^+$. Since $\omega^+$ is upper triangular we have $\sum_{t}\eta_{sqt}=0$ if $s>q$. In particular, $\sum_{t}\eta_{s1t}=0, s>1.$ Since the $\eta_{sqt}$ are nonnegative, we obtain $\eta_{s1t}=0$ for all $s>1$ and all $t$. From this and $\pi^+=\eta^1$ we have \[\pi_{11}+p^d=\sum_s\eta_{s11}=\eta_{111}.\]
	Hence $\eta_{111} \ge p^d$. Using this one easily checks that we have a map 
	\[\{\eta \in \La(n,n,n;r'):\eta^1=\pi^+, \eta^3=\omega^+\} \to \{\theta \in \La(n,n,n;r):\theta^1=\pi, \theta^3=\omega\}, \eta \mapsto \eta^{-},\]
	where $\eta^{-}_{111} =\eta_{111}-p^d$ and $\eta^{-}_{sqt} =\eta_{sqt}$ for all $(s,q,t) \neq (1,1,1)$. Also it is clear from the definition that $(\theta^{+})^{-}=\theta$ and $(\eta^{-})^{+}=\eta$. Hence the map of the lemma is a bijection.
	
	The proof when $\pi$ is lower triangular is similar and thus omitted. (Alternately, one may use the first case and the algebra anti-automorphism $S \to S$ defined by $\xi_\omega \mapsto \xi_{{\omega}^t}$, where ${\omega}^t$ denotes the transpose of ${\omega}^t$.) \end{proof}
In general the map $S \to S', \xi_\omega \mapsto \xi_{\omega ^+},$ is not a homomorphism of algebras. In the next lemma, however, we have two instances where the product rule is `well behaved' with respect to this map. A third instance will be treated in Lemma 4.11.

We note that thus far the assumption on the characteristic $p>0$ of $K$ has not been used. In particular, Lemmas 4.6 and 4.7 are valid for any positive integer in place of $p^d$. Case (3) of the next lemma is the first instance where the characteristic $p>0$ assumption is needed.
\begin{lemma}
	Let $\omega, \pi \in \La(n,n;r)$ such that $\omega$ is upper triangular and let $\theta \in \La(n,n,n;r)$  such that $\theta^1=\pi, \theta^3=\omega.$ \begin{enumerate}
		\item If $\pi$ is upper triangular, then $[\theta]=[\theta^+].$
		\item If $\pi$ is lower triangular, $\omega^1 \unrhd \la$ and $p^d > r-\la_1$, then $[\theta] \equiv [\theta^+] \mod p.$
		\item Under the assumptions of either (1) or (2) above, we have the equivalence \begin{equation*}
			\xi_\omega \xi_\pi = \sum_{\substack{\theta \in \La(n,n,n;r) \\ \theta^3=\omega, \theta^1=\pi}}[\theta]\xi_{\theta^2} \Leftrightarrow \xi_{{\omega}^+} \xi_{{\pi}^+} = \sum_{\substack{{\theta}^+ \in \La(n,n,n;r') \\ {({\theta}^+)}^3={\omega}^+, {({\theta}^+)}^1={\pi}^+}}[\theta]\xi_{({{\theta}^+})^2}.
		\end{equation*}
	\end{enumerate}
\end{lemma}
\begin{proof} (1) Since $\theta_{sqt}=\theta^{+}_{sqt}$ if $(s,q,t) \neq (1,1,1)$ it suffices by eqn. (4.3) to show that \[\tbinom{(\theta^2)_{11}}{\theta_{111},\theta_{121},\dots,\theta_{1n1}}=\tbinom{((\theta^+)^2)_{11}}{(\theta^+)_{111},(\theta^+)_{121},\dots,(\theta^+)_{1n1}}.\]
	Since $\theta^1=\pi$ and $\pi$ is upper triangular, we have $\sum_{s}\theta_{sq1}=0$ for all $q>1$. Since the $\theta_{sq1}$ are nonnegative, we obtain $\theta_{121}= \theta_{131}= \cdots = \theta_{1n1}=0$ and the above multinomial coefficients are equal to 1.
	
	(2) By the first line of the previous paragraph, it suffices to show that \[\tbinom{(\theta^2)_{11}}{\theta_{111},\theta_{121},\dots,\theta_{1n1}} \equiv \tbinom{(\theta^2)_{11}+p^d}{\theta_{111}+p^d,\theta_{121},\dots,\theta_{1n1}} \mod p.\] We will show that $p^d > (\theta^2)_{11}-\theta_{111}$ and hence the above congruence follows from Lemma 4.1(2). As in the second paragraph of the proof of Lemma 4.7 (for $\theta$ in place of $\eta$), we have $\theta _{111}=\pi_{11}.$ Using that $\pi$ is lower triangular and the hypotheses $\pi^2 =\omega^1$ and $\omega^1 \unrhd \la$, we have \[\pi_{11}=(\pi^2)_1=(\omega^1)_1\ge \la_1.\] Hence $\theta_{111} \ge \la_1$.  Since $r= \sum_{s,q,t}\theta_{sqt}$ and the $ \theta_{sqt}$ are nonnegative, we have $r \ge (\theta^2)_{11}$. Hence $p^d >r-\la_1 \ge (\theta^2)_{11}-\la_1 \ge (\theta^2)_{11}-\theta_{111}.$
	
	(3) This follows from parts (1) and (2) of the lemma and Lemma 4.7. \end{proof}

Let $\al \in \La(n,r)$. In \cite[Theorem A.4]{SY} it was shown that the map \begin{equation}
	\psi_\al : D(\al) \to S\xi_{\al}, \; e^{(\pi)} \mapsto \xi_{\pi},
\end{equation}
is an isomorphism of $GL_n(K)$-modules, where $\pi \in \La(n,n;r)$, $\pi^1 = \al$ and \[e^{(\pi)}=1^{(\pi_{11})}2^{(\pi_{21})}\cdots n^{(\pi_{n1})} \otimes \cdots \otimes 1^{(\pi_{1n})}2^{(\pi_{2n})}\cdots n^{(\pi_{nn})}.\] In the next lemma, we identify the map of Lemma 3.3(2) under the above isomorphism.

Recall the following definition from subsection 2.3. For $T \in \mathrm{RSST}_{\al}(\mu)$, \[T=\begin{matrix*}[l]
	1^{(a_{11})}2^{(a_{21})}\cdots n^{(a_{n1})}\\
	\cdots\\
	1^{(a_{1n})}2^{(a_{2n})}\cdots n^{(a_{nn})},\end{matrix*}\]
define $\omega(T)=(\al_{ij})$. Then $\omega(T)^1=\mu$ and $\omega(T)^2=\al$. According to  \cite[p. 2593]{SY}, this defines a bijection between the sets $T \in \mathrm{RSST}_{\al}(\mu)$ and $\Omega(\al,\mu)$, where for any $\beta \in \La(n;r)$, \[\Omega(\al,\beta)=\{\omega \in \La(n,n;r): \omega^1 =\beta , \omega^2=\al\}.\]

\begin{remark} With the previous notation we note that if  $T \in \mathrm{RSST}_{\al}(\mu)$ is semistandard, then $\omega(T)$ is lower triangular.\end{remark}

For $\mu \in \Lambda^+(n,r)$ we write $\pi_{\D(\mu)}: S\xi_{\mu} \to \D(\mu)$ and $\pi_{L(\mu)}: S\xi_{\mu} \to L(\mu)$ for the indicated natural projections. Recall we have the maps $\phi_\omega$ of Definition 2.3.

\begin{lemma} Let $\al \in \La(n,r)$. \begin{enumerate} \item Let $\beta \in \La(n,r)$ and $\omega \in \Omega(\al,\beta).$ Then under the isomorphism \[\Hom_S(D(\al),D(\beta)) \simeq \Hom_S(S{\xi_\al},S{\xi_{\beta}})\] induced by the map in (4.8), the image of $\phi_\omega$ is the map  $R_{\xi_{\omega}} :S\xi_{\al }\to S\xi_{\beta}$ defined as right multiplication by $\xi_{\omega}$ .
		\item Let $\mu \in \Lambda^+(n,r)$ such that $\mu_2 \le \al_1$. Then under the isomorphism 
		\[\Hom_S(S\xi_{\al}, \D(\mu)) \to \Hom_{S'}(S'\xi_{\al^+}, \D(\mu^+))\] induced by Lemma 3.3(2) and (4.8), the image of $\pi_{\D(\mu)} \circ R_{\xi_{\omega}}$, where $\omega \in \Omega(\al, \mu)$, is $ \pi_{\D(\mu^+)} \circ R_{\xi_{\omega^+}}$.
	\end{enumerate}
\end{lemma}
\begin{proof} (1)
	We need to check that the following diagram commutes. 
	\begin{center}
		\begin{tikzcd}
			D(\al)\arrow{d}{\psi_\al} \arrow{rr}{\phi_\omega}
			&& D(\beta) \arrow{d}{\psi_\beta} \\
			S\xi_\al \arrow{rr}{R_{\xi_{\omega}}}
			&& 	S\xi_\beta
		\end{tikzcd}
	\end{center}	
	Since this is a diagram of $S$-homomorphisms and as a $S$-module $D(\al)$ is cyclic generated by $e^{(\al)}$, it suffices to show that $\psi_\beta(\phi_\omega(e^{(\al)}))=	R_{\xi_{\omega}} (\psi_\al (e^{(\al)})).$ We have \[R_{\xi_{\omega}} (\psi_\al (e^{(\al)}))=R_{\xi_{\omega}}(\xi_\al) = \xi_\al\xi_{\omega}=\xi_{\omega},\]
	the last equality owing to eqn. (4.4). On the other hand we have \[\psi_\beta(\phi_\omega(e^{(\al)}))=\psi_\beta(e^{(\omega)})=\xi_{\omega}.\]
	
	(2) This follows from part (1) for $\beta = \mu$.
\end{proof}
\subsection{Proof of Theorem 1.1(1)}We will show that the complexes of vector spaces $\Hom_S(B_*(\la), \D(\mu))$ and $\Hom_{S'}(B_*(\la^+), \D(\mu^+))$ are isomorphic, where $B_*(\la)$ and $B_*(\la^+)$ are the projective resolutions of $\D(\la)$ and $\D(\la^+)$, respectively, of Santana and Yudin that we recalled in subsection 4.2.

From Remark 4.5 and Lemma 4.6 it follows that the map of Lemma 4.10(2) yields an isomorphism  \[\Hom_S(B_k(\la), \D(\mu)) \to \Hom_{S'}(B_k(\la^+), \D(\mu^+))\] for every $k$. Adopting the notation of Remark 4.4, it suffices to show that the following diagrams commute

\begin{center}
	\begin{tikzcd}
		\Hom_S(M[\omega_2,\dots,\omega_k],\Delta(\mu))\arrow{d} \arrow{r}{}
		& \Hom_S(M[\omega_1,\dots,\omega_k],\Delta(\mu)) \arrow{d}\\ \arrow{r}{}
		\Hom_{S'}(M[\omega_2 ^+,\dots,\omega_k ^+],\Delta(\mu^+)) 
		& \Hom_{S'}(M[\omega_1 ^+,\dots,\omega_k ^+],\Delta(\mu^+))
	\end{tikzcd}
\end{center}
and
\begin{center}
	\begin{tikzcd}
		\Hom_S(M[\omega_1,\dots,\omega_i\omega_{i+1},\dots,\omega_k],\Delta(\mu))\arrow{d} \arrow{r}{}
		& \Hom_S(M[\omega_1,\dots,\omega_k],\Delta(\mu)) \arrow{d}\\ \arrow{r}{}
		\Hom_{S'}(M[\omega_1 ^+,\dots,\omega_i ^+ \omega_{i+1}^+,\dots,\omega_k ^+],\Delta(\mu ^+)) 
		& \Hom_{S'}(M[\omega_1 ^+,\dots,\omega_k ^+],\Delta(\mu^+))
	\end{tikzcd}
\end{center}
where the vertical maps are the isomorphisms of Lemma 4.10(2) and the horizontal maps are the indicated components of the differentials of the complexes $\Hom_S(B_*(\la), \D(\mu))$ and $\Hom_S(B_*(\la^+), \D(\mu^+)).$

Consider the first diagram. We know that $\Hom_S(M[\omega_2,\dots,\omega_k],\Delta(\mu))$ is generated by maps of the form $f=\pi_{\D(\mu)}\circ R_{\xi_\omega}$, where $\omega$ is lower triangular (Remark 4.9). We have \[\Hom_S(R_{\xi_{\omega_1}}, \D(\mu))(f)=\pi_{\D(\mu)} \circ R_{\xi_{\omega_1} \xi_\omega}.\] Using eqn. (4.2) and applying the vertical isomorphism, we see that that the image of $f$ in the clockwise direction is the map 
\[\pi_{\D(\mu ^+)} \circ ( \sum_{\theta}[\theta ^2] R_{\xi_{(\theta ^2)^+}}): M[\omega_1 ^+,\dots,\omega_k ^+] \to \D(\mu ^+),\]
where the sum ranges over all $\theta \in \La(n,n;r)$ such that $\theta ^3 = \omega_1, \theta ^1=\omega$. In the counterclockwise direction, the image of $f$ is the map \[\pi_{\D(\mu ^+)} \circ R_{\xi_{(\omega_1)^+}\xi_{\omega ^+}}: M[\omega_1 ^+,\dots,\omega_k ^+] \to \D(\mu ^+).\] Since $\omega_1$ is upper triangular (cf. Definition 4.2) and $\omega$ is lower triangular, we may apply Lemma 4.8(3) to conclude that the two images coincide and the diagram commutes.

Consider the second diagram and Remark 4.4. Since both $\omega_i$ and $\omega_{i+1}$ are upper triangular, we may apply Lemma 4.8(3) to conclude that it suffices to show that the following diagram commutes
\begin{center}
	\begin{tikzcd}
		\Hom_S(M[\omega_1,\dots,\theta^2,\dots,\omega_k],\Delta(\mu))\arrow{d} \arrow{r}{}
		& \Hom_S(M[\omega_1,\dots,\omega_k],\Delta(\mu)) \arrow{d}\\ \arrow{r}{}
		\Hom_{S'}(M[\omega_1 ^+,\dots,(\theta ^2)^+,\dots,\omega_k ^+],\Delta(\mu ^+)) 
		& \Hom_{S'}(M[\omega_1 ^+,\dots,\omega_k ^+],\Delta(\mu^+))
	\end{tikzcd}
\end{center}
for each $\theta \in \La(n,n,n;r)$ such that  $\theta^3=\omega_i$ and $\theta^1=\omega_{i+1}.$ But this is clear as both horizontal maps are multiplication by $(-1)^i[\theta ^2] \in K$.
\subsection{Weight subspaces of simples modules}
For the proof of Theorem 1.1(2), we will need an analog of Lemma 4.10(2) for the simple modules $L(\mu)$ and $L(\mu^+)$ in place of the corresponding Weyl modules. Recall that $\D(\la)$ has a unique maximal submodule $\rad\D(\la)$ and that $L(\la)=\D(\la)/\rad\D(\la)$, see \cite[Part II, 2.14]{Jan}.
\begin{lemma}  Let $\al \in \La(n;r)$ such that $\al \unrhd \la$. Suppose  $p^d > r-\la_1$. \begin{enumerate}
		\item Let $\omega, \pi \in \La(n,n;r)$ such that $\pi$ is lower triangular and $\omega^1 = \al \unrhd \la$. We have the equivalence\begin{equation*}
			\xi_\omega \xi_\pi = \sum_{\substack{\theta \in \La(n,n,n;r) \\ \theta^3=\omega, \; \theta^1=\pi}}[\theta]\xi_{\theta^2} \; \Leftrightarrow \; \xi_{{\omega}^+} \xi_{{\pi}^+} = \sum_{\substack{{\theta}^+ \in \La(n,n,n;r') \\ {({\theta}^+)}^3={\omega}^+, \;  {({\theta}^+)}^1={\pi}^+}}[\theta]\xi_{({{\theta}^+})^2}.
		\end{equation*}
		\item Suppose $\la_1 \ge \frac{r}{2}$   and let $g \in \Hom_S(S\xi_{\al}, \D(\mu))$. We have the equivalence \[Img \subseteq rad\D(\mu)  \Leftrightarrow Img^+ \subseteq rad\D(\mu^+),\] where $g^+$ is the image of $g$ under the isomorphism of Lemma 4.10 (2).
	\end{enumerate}
\end{lemma}
\begin{proof} (1)  According to Lemma 4.7, it suffices to show that $[\theta]=[\theta^+]$ in $K$ for all $\theta \in \La(n,n,n;r)$  such that $\theta^1=\pi$ and $\theta^3=\omega$ and hence it suffices to show that \[\tbinom{(\theta^2)_{11}}{\theta_{111},\theta_{121},\dots,\theta_{1n1}} \equiv \tbinom{(\theta^2)_{11}+p^d}{\theta_{111}+p^d,\theta_{121},\dots,\theta_{1n1}} \mod p.\]
	Since $\theta^3=\omega$ and $\omega^1=\al$, we have \[\theta_{121}+\dots+\theta_{1n1} \le \al_2+\dots+a_n=r-\al_1 \]
	and since $\omega^1 \unrhd \la$ we have $\al_1 \ge \la_1.$ Hence $\theta_{121}+\dots+\theta_{1n1} \le r-\la_1<p^d.$ The desired result follows from Lemma 4.1(1).
	
	(2) Since $\D(\mu)$ and $\D(\mu^+)$ are highest weight modules with corresponding highest weights $\mu$ and $\mu^+$ respectively, it suffices to show that \begin{equation}(S\xi_{\al})_{\mu} \subseteq \ker g  \Leftrightarrow (S'\xi_{\al^+})_{\mu^+} \subseteq \ker g^+.\end{equation}
	
	Since $\Hom_S(S\xi_{\al}, \D(\mu))$ is generated by the maps $ \pi_{\D(\mu)} \circ R_{\xi_{\pi}}$,  where $\pi \in \La(n,n;r)$ satisfies $\pi^1=\mu$, $\pi^2=\al$ and is lower triangular (cf. Remark 4.9), write \[g=\sum_{j}d_j\pi_{\D(\mu)} \circ R_{\xi_{\pi_j}}\] accordingly, where $d_j \in K$. The hypothesis $\mu_2 \le \al _1$ of Lemma 4.10(2) holds because we have $\la_1 \ge \frac{r}{2}\Rightarrow  \al_1 \ge \la_1 \ge \frac{r}{2} = \frac{1}{2}{(\mu_1+\mu_2+\dots+\mu_n)} \ge \frac{1}{2}{(\mu_2+\mu_2)} =\mu_2$ and thus we obtain  \[g^+=\sum_{j}d_j\pi_{\D(\mu^+)} \circ R_{\xi_{(\pi_j)^+}}.\]
	
	Assume $(S\xi_{\al})_{\mu} \subseteq \ker g$ and suppose $y \in (S'\xi_{\al^+})_{\mu^+}$. The vector space $(S'\xi_{\al^+})_{\mu^+}$ is spanned by the $\xi_\rho$, where $\rho \in \La(n,n;r')$, $\rho^1=\al^+$ and $\rho^2=\mu^+$. We claim that the hypothesis $\la_1 \ge \frac{r}{2}$ implies that each such $\rho$ is of the form $\rho=\omega^+$, where $\omega \in \La(n,n;r)$, $\omega^1=\al$ and $\omega^2=\mu$. Indeed, using $\rho^2=\mu^+$ we have \[\mu_1+p^d=\rho_{11}+\rho_{12}+\dots+\rho_{1n}\le \rho_{11} +\al_2 +\dots+\al_n=\rho_{11}+r-\al_1,\] the inequality owing to $\rho^1=\al^+$. Thus $\rho_{11}-p^d \ge \mu_1+\al_1-r \ge 2\la_1-r \ge 0$ because $\mu \unrhd \la$ and $\al \unrhd \la$. Thus $\rho=\omega^+$ for some $\omega \in \La(n,n;r)$ such that $\omega^1=\al$ and $\omega^2=\mu$. Hence $y=\sum_{i}c_i\xi_{(\omega_i)^+}$ for some $c_i \in K$ and  $\omega_i \in \La(n,n;r)$ such that $(\omega_i)^1=\al$ and $(\omega_i)^2=\mu$.  Substituting we obtain \begin{equation}g^+(y)=\sum_{i,j}c_id_j\pi_{\D(\mu^+)}(\xi_{(\omega_i)^+}\xi_{(\pi_j)^+}).\end{equation}
	Letting $x=\sum_{i}c_i\xi_{\omega_i}$ we note that $x \in (S\xi_{\al})_{\mu}$ and thus by hypothesis we have $g(x)=0$ in $\D(\mu)$, that is \begin{equation}\sum_{i,j}c_id_j\pi_{\D(\mu)}(\xi_{\omega_i}\xi_{\pi_j})=0.\end{equation} In the above equation each $\pi_j$ is lower triangular and thus we may apply the first part of the lemma to (4.10) and (4.11) to obtain
	
	\begin{equation}g^+(y)=\sum_{i,j}c_id_j\pi_{\D(\mu^+)}\Big(\sum_{\substack{{\theta}^+ \in \La(n,n,n;r') \\ {({\theta}^+)}^3={(\omega_i)}^+, \;  {({\theta}^+)}^1={(\pi_j)}^+}}[\theta]\xi_{(\theta ^+)^2}\Big).\end{equation}
	and 
	\begin{equation}\sum_{i,j}c_id_j\pi_{\D(\mu)}\Big(\sum_{\substack{\theta \in \La(n,n,n;r) \\ \theta^3=\omega_i, \; \theta^1=\pi_j}}[\theta]\xi_{\theta ^2}\Big)=0.\end{equation}
	Now each element $\pi_{\D{(\mu)}}(\xi_{\theta^2}) \in \D(\mu)$ in eqn. (4.13) is of the form $[U_{\theta}]$, where $U_{\theta} \in \mathrm{Tab}_{\mu}(\mu)$. Hence we may apply Corollary 3.4 to  eqns. (4.12) and (4.13) to conclude that $g^+(y)=0$.
	
	Conversely, assume $(S'\xi_{\al^+})_{\mu^+} \subseteq \ker g^+$ and suppose $x \in (S\xi_{\al})_{\mu}$. This is similar to the previous case. The vector space $(S\xi_{\al})_{\mu}$ is spanned by the $\xi_\omega$, where $\omega \in \La(n,n;r)$, $\omega^1=\al$ and $\omega^2=\mu$. Write $x=\sum_{i}c_i\xi_{\omega_i}$ accordingly for some $c_i \in K$ and let $y=\sum_{i}c_i\xi_{{\omega_i}^+} \in (S'\xi_{\al^+})_{\mu^+}$. From the hypothesis $g^+(y)=0$ we obtain 
	\[\sum_{i,j}c_id_j\pi_{\D(\mu^+)}(\xi_{(\omega_i)^+}\xi_{(\pi_j)^+})=0.\]
	On the other hand, upon substitution we have \[g(x)=\sum_{i,j}c_id_j\pi_{\D(\mu)}(\xi_{\omega_i}\xi_{\pi_j}).\]
	Since each $\pi_j$ is lower triangular we may apply the first part of the lemma and Corollary 3.4 to conclude that $g(x)=0$ in $\D(\mu^+)$.	
\end{proof}

\begin{proposition}  Suppose  $p^d > r-\la_1$ and $\la_1 \ge \frac{r}{2}$. Let $\al \in \Lambda(n,r)$ such that $\al \unrhd \la$. Then there is an isomorphism  $\Hom_S(S\xi_{\al}, L(\mu)) \to \Hom_{S'}(S'\xi_{\al^+}, L(\mu^+))$ such that \[\pi_{L(\mu)} \circ R_{\xi_{\omega}} \mapsto \pi_{L(\mu ^+)} \circ R_{\xi_{\omega ^+}}\] for all $\omega \in \Omega(\al, \mu)$.
\end{proposition}
\begin{proof} Let $f \in \Hom_S(S\xi_{\al}, L(\mu))$. Since $S\xi_{\al}$ is a projective $S$-module, there is a $f_1 \in \Hom_S(S\xi_{\al}, \D(\mu))$ such that $f=pr_{\mu}\circ f_1$, where $pr_{\mu}:\D(\mu) \to L(\mu)$ is the natural projection. Define $f^+ = pr_{\mu^+} \circ f_1^+$, where $f_1^+ \in \Hom_{S'}(S'\xi_{\al^+}, \D(\mu^+))$ is the image of $f_1$ under the isomorphism of Lemma 4.10(2).
	
	In order to show that the map $f \to f^+$ is well defined, suppose $pr_{\mu} \circ g=0$, where $g \in \Hom_S(S\xi_{\al}, \D(\mu))$. Then $Img \subseteq rad\D(\mu)$. By Lemma 4.11(2), $Img^+ \subseteq rad\D({\mu}^+)$ and hence $pr_{{\mu}^+} \circ g^+=0.$
	
	In light of Lemma 4.11(2), the previous argument can be reversed and thus the map $f \to f^+$ is 1-1. To show that it is onto, let $h \in \Hom_{S'}(S'\xi_{\al^+},L(\mu^+))$. Since $S'\xi_{\al^+}$ is a projective $S'$-module, there is a $h_1 \in \Hom_S'(S'\xi_{\al^+}, \D(\mu^+))$ such that $h=pr_{\mu^+}\circ h_1$. By Lemma 4.10(2), $h_1=g^+$ for some $g \in \Hom_S(S\xi_{\al}, \D(\mu))$.
\end{proof}
\begin{remark}
	Denote by $K_{\al,\mu}$ the dimension of the vector space $\Hom_S(S\xi_{\al}, L(\mu))$. These are known as $p$-Kostka numbers. Under the hypotheses of the previous proposition we have that $K_{\al, \mu} = K_{\al^+, \mu ^+}$. This was proved using different means by Henke in \cite[Corollary 6.2]{He} under the weaker hypothesis $p^d>\la(\al)_2$, where $\la(\al)=(\la(\al)_1,\dots,\la(\al)_n)$ is the partition obtained by rearranging the parts of $\al$. However, Proposition 4.12 says a bit more in that we have an explicit (well defined) isomorphism. This is needed in the proof that follows.
\end{remark}
\subsection{Proof of Theorem 1.1(2)} 

The proof that the complexes $\Hom_S(B_*(\la), L(\mu))$ and $\Hom_{S'}(B_*(\la^+), L(\mu^+))$ are isomorphic is similar to the proof in subsection 4.4; one uses Proposition 4.12 in place of Lemma 4.10(2) and the following remark in place Remark 4.9.
\begin{remark} If $\al \unrhd \la$,  $\mu \unrhd \la$ and $\la_1 \ge \frac{r}{2}$, then $\Hom_S(S\xi_a, L(\mu))$ is generated by maps of the form $S\xi_{a} \xrightarrow{R_{\xi_\omega}} \D(\mu) \xrightarrow{pr_{\mu}} L(\mu)$, where is $\omega$ is lower triangular. Indeed, by Remark 4.9, $\Hom_S(S\xi_a, \D(\mu))$ is generated by maps of the form $S\xi_{a} \xrightarrow{R_{\xi_\omega}} \D(\mu)$, where $\omega$ is lower triangular. Since $S\xi_{a}$ is a projective $S$-module, we have a surjective map $\Hom_S(S\xi_a, \D(\mu)) \to \Hom_S(S\xi_a, L(\mu)) $ induced from the natural projection $\D(\mu) \xrightarrow{pr_{\mu}} L(\mu)$.
\end{remark}
\section{Representations of the symmetric group}
In this section we assume that $p>2$. For $\lambda$ a partition of $n$ we denote by $S^{\lambda}$ the Specht module corresponding to $\lambda$ for the group algebra $K\mathfrak{S}_n$ of the symmetric group $\mathfrak{S}_n$. By a classical result of Carter and Luzstig \cite{CL},  the vector spaces $\Hom_S(\Delta(\nu), \Delta(\nu'))$ and $\Hom_{\mathfrak{S}_n}(S^{\nu'}, S^{\nu})$ are isomorphic for all partitions $\nu, \nu'$ of $n$, where $S=S_K(n,n)$. Hence from Theorem 1.1(1) we obtain the following result.

\begin{corollary}Let $\la=(\la_1,\dots,\la_n)$ and $\mu=(\mu_1,\dots,\mu_n)$  be partitions of $n$. If $p^d >n-\la_1$ and $\mu_2\le \la_1$, then \[\Hom_{\mathfrak{S}_n}(S^{\mu}, S^{\la})\simeq \Hom_{\mathfrak{S}_{n'}}(S^{\mu^+}, S^{\la^+}).\]
\end{corollary}
This provides an answer to a question posed by Hemmer, \cite{Hem2} Problem 5.4.

By a result of Kleshchev and Nakano \cite[Theorem 6.4b(ii)]{KN} we obtain more generally from Theorem 1.1(1) the following.
\begin{corollary}Suppose $p>3$ and $ 0 \le i \le p-2$. Let $\la=(\la_1,\dots,\la_n)$ and $\mu=(\mu_1,\dots,\mu_n)$  be partitions of $n$. If $p^d >n-\la_1$ and $\mu_2\le \la_1$, then \[\Ext^i_{\mathfrak{S}_n}(S^{\mu}, S^{\la})\simeq \Ext^i_{\mathfrak{S}_{n'}}(S^{\mu^+}, S^{\la^+}).\]
\end{corollary}

Let $H^i(\mathfrak{S}_n, 
S^{\la })$ be the degree $i$ cohomology group of $\mathfrak{S}_n$ with coefficients in $S^{\lambda}$. By a result of Kleshchev and Nakano \cite[Corollary 6.3b(iii)]{KN}, \[H^i(\mathfrak{S}_n, S^{\la }) \simeq \Ext^i_S( \nabla(n), \nabla (\la)   ),\] if $0 \le i \le 2p-4$, where $\nabla (\la)$ is the costandard module of $G$ corresponding to $\la$. By contravariant duality, \cite[eqn. 4.13 (3)]{Jan}, we have \[\Ext^i_S( \nabla(n), \nabla (\la)) \simeq \Ext^i_S( \D(\la), \D(n)).\] Therefore from Theorem 1.1(1) we obtain the following corollary.

\begin{corollary}
	Let $\la$ be a partition of $n$. If $p^d >n-\la_1$, then \[H^i(\mathfrak{S}_n, S^{\la }) \simeq H^i(\mathfrak{S}_{n+p^d}, S^{\la^+ }), \] for $0 \le i \le 2p-4$. 
\end{corollary}
\begin{remarks} (1) It is well known that the case $i=0$ in the above corollary holds. This follows from the explicit description of $H^0(\mathfrak{S}_n, S^{\la })$ given by James  \cite[Theorem 24.4]{Jam}. The case $i=1$ under the slightly stronger assumption $p^d>n$ is due to Hemmer \cite[Theorem 7.1.8]{Hem1} with a completely different proof. An explicit description of $H^1(\mathfrak{S}_n,S^{\la })$ was obtained by Donkin and Geranios in \cite{DG}.\\
	(2) A related but different result has been obtained by Nagpal and Snowden \cite[Theorem 5.1]{NS}. For example, their theorem assumes $\la$ (denoted by $\mu[n]$) satisfies $\la_1-\la_2 \ge \la_2+\dots+\la_n$. Also, an upper bound on $i$ is determined by $\la_1$ in \cite{NS}, whereas in the above corollary an upper bound is determined by $p$.\\(3) The above corollary provides an answer to the special case of Problem 8.3.1 of \cite{Hem1} corresponding to $c=(1)$.
\end{remarks}
\section{Better bound on the period}
The purpose of this section is to observe that in two examples, the bound $p^d > r-\la_1$ in the statement of Theorem 1.1(1) may be improved. This is done using specific projective resolutions for these examples that involve fewer terms than those of \cite{SY}. The proofs are similar (and simpler) but we provide details for completeness. From Section 4 we use only Lemma 4.1.
\subsection{First example: Hom spaces}
\begin{theorem}Let $\la=(\la_1,\dots,\la_n)$, $\mu=(\mu_1,\dots,\mu_n)$  be partitions of $r$. Suppose $p^d >\min\{\la_2, \mu_1-\la_1\}$ and $\mu_2\le \la_1$. Then   $\Hom_S(\Delta(\la), \Delta(\mu)) \simeq \Hom_{S'}(\Delta(\la^+), \Delta(\mu^+)).$
\end{theorem}

\begin{example}We note that if either of the assumptions $p^d >\min\{\la_2, \mu_1-\la_1\}$ and $\mu_2 \le \la_1$ of Theorem 6.1 is relaxed, then the conclusion is not necessarily true, as the following examples show. First, let $p=3, \la=(8,3), \mu=(11)$ and $d=1$ so that $\la^+=(11,3)$ and $\mu^+=(14)$. Here we have $p^d =\min\{\la_2, \mu_1-\la_1\}$. The dimensions of the spaces $\Hom_S(\Delta(\la), \Delta(\mu)), \Hom_{S'}(\Delta(\la^+), \Delta(\mu^+))$ are 1 and 0 respectively. Second, let $p=3, \la=(1,1,1,1), \mu=(2,2)$ and $d=1$ so that $\la^+=(4,1,1,1)$ and $\mu^+=(5,2)$. The dimensions of the corresponding Hom spaces are 1 and 0 respectively. These dimensions follow, for example, from  Theorem 3.1 of \cite{MS1} since in each case a hook is involved.
	
	We note that the previous theorem was proven in the unpublished manuscript \cite{MS4} which is superseded by the present paper.
	
\end{example}

\noindent\textit{Proof of Theorem 6.1.} Recall from \cite[Theorem II.3.16]{ABW} that we have the following beginning of a resolution of $\Delta(\la)$, \[
\sum_{i=1}^{n-1}\sum_{t=1}^{\lambda_{i+1}}M_i(t) \xrightarrow{\square_{\la}} D(\lambda) \xrightarrow{\pi_{\D(\la)}} \Delta(\lambda) \to 0,\]
where the restriction of $ \square_{\la} $ to the summand $M_i(t)=D(\lambda_1,\dots,\lambda_i+t,\lambda_{i+1}-t,\dots,\lambda_n)$ is the composition
\[
M_i(t) \xrightarrow{1\otimes\cdots \otimes \Delta \otimes \cdots \otimes 1}D(\lambda_1,\dots,\lambda_i,t,\lambda_{i+1}-t,\dots,\lambda_m)\xrightarrow{1\otimes\cdots \otimes m \otimes \cdots \otimes 1} D(\lambda),
\]
where $\Delta:D(\lambda_i+t) \to D(\lambda_i,t)$ and $m:D(t,\lambda_{i+1}-t) \to D(\lambda_{i+1})$ are the indicated components of the comultiplication and multiplication respectively of the Hopf algebra $DV$. Likewise we have the exact sequence
$\sum_{i=1}^{n-1}\sum_{t=1}^{\lambda_{i+1}}M^+_i(t) \xrightarrow{\square_{\la^+}} D(\lambda^+) \xrightarrow{d'_{\lambda^+}} \Delta(\lambda^+) \to 0,$
where $M^+_i(t)=D(\lambda_1^+,\dots,\lambda_i+t,\lambda_{i+1}-t,\dots,\lambda_n)$

By left exactness of the functors $\Hom_{S}(-,\Delta(\mu))$ and $\Hom_{S'}(-,\Delta(\mu^+))$, in order to prove the theorem it suffices to show that the following diagram commutes
\begin{center}
	\begin{tikzcd}
		\Hom_S(D(\lambda),\Delta(\mu))\arrow{d}{\psi} \arrow{rrrr}{\pi_{i,t} \circ \Hom_S(\square_{\la},\Delta({\mu}))}
		&&&& \Hom_S(M_i(t),\Delta(\mu)) \arrow{d}{\psi_{i,t}} \\ \arrow{rrrr}{\pi^+_{i,t} \circ \Hom_{S'}(\square_{\la^+},\Delta({\mu^+}))}
		\Hom_{S'}(D(\lambda^+),\Delta(\mu^+)) 
		&&&& \Hom_{S'}(M^+_i(t),\Delta(\mu^+))
	\end{tikzcd}
\end{center}
for all $i=1,\dots,n-1$ and $t=1,\dots,\la_{i+1}$, where the vertical maps are isomorphisms from Lemma 3.3(2) and $\pi_{i,t}$ (respectively, $\pi^+_{i,t}$) is the natural projection $\Hom_S(\sum_{i=1}^{n-1}\sum_{t=1}^{\lambda_{i+1}}M_i(t),\Delta(\mu)) \to \Hom_S(M_i(t),\Delta(\mu))$ (respectively, \\ $\Hom_{S'}(\sum_{i=1}^{n-1}\sum_{t=1}^{\lambda_{i+1}}M^+_i(t),\Delta(\mu^+)) \to \Hom_{S'}(M^+_i(t),\Delta(\mu^+))$ ).

For $i=1,\dots,n-1$ and $t=1,\dots,\la_{i+1}$, let $
x_{i,t}=1^{(\la_1)} \otimes \cdots \otimes  n^{(\la_n)} \in M_i(t)$ and $x^+_{i,t}=1^{(\la_1+p^d)} \otimes \cdots  \otimes n^{(\la_n)} \in M^+_{i}(t).$
We know that $M_i(t)$ is a cyclic $S$-module with generator $x_{i,t}$ and $M^+_i(t)$ is a cyclic $S'$-module with generator $x^+_{i,t}$.

Let $ T \in \st$. Then \[ T=\begin{matrix*}[l]
	1^{(\la_1)}2^{(a_{12})}3^{(a_{13})} \cdots n^{(a_{1n})} \\
	2^{(a_{22})}3^{(a_{23})} \cdots n^{(a_{2n})} \\
	\cdots\\
	n^{(a_{nn})}  \end{matrix*}, \]
where $\sum_{i}a_{ij}=\la_j$ for $j=2,\dots,n$. We know that the set $\{\pi_{\D(\mu)} \circ\phi_T: T 
\in \st \}$ generates $\Hom_S(D(\la), \Delta(\mu))$ by Proposition 2.4. We distinguish two cases.

\textit{Case $i=1$.} Consider $x_{1,t}$, where $1 \le t \le \la_2$. By computing we see the image of $x^{+}_{1,t}$ under the map $\psi_{1,t}\circ \pi_{1,t} \circ \Hom_S(\square_{\la},\Delta({\mu}))(\pi_{\D(\mu)} \circ\phi_T)$ is  \begin{equation}\sum_{j_1+j_2=t}\tbinom{\la_1+j_1}{j_1}\begin{bmatrix*}[l]
		1^{(\la_1+p^d+j_1)}2^{(a_{12}-j_1)}3^{(a_{13})} \cdots n^{(a_{1n})} \\
		1^{(j_2)}2^{(a_{22-j_2})}3^{(a_{23})} \cdots n^{(a_{2n})} \\
		\cdots\\
		n^{(a_{nn})}  \end{bmatrix*}.
\end{equation}
A similar computation in the other direction of the diagram yields that the image of   $x^{+}_{1,t}$ under the map $\pi_{1,t}^+ \circ \Hom_S(\square_{\la^+},\Delta({\mu^+}))\circ \psi(\pi_{\D(\mu)} \circ\phi_T)$ is  \begin{equation}\sum_{j_1+j_2=t}\tbinom{\la_1+p^d+j_1}{j_1}\begin{bmatrix*}[l]
		1^{(\la_1+p^d+j_1)}2^{(a_{12}-j_1)}3^{(a_{13})} \cdots n^{(a_{1n})} \\
		1^{(j_2)}2^{(a_{22-j_2})}3^{(a_{23})} \cdots n^{(a_{2n})} \\
		\cdots\\
		n^{(a_{nn})}  \end{bmatrix*}.
\end{equation} By Lemma 3.2(i) we may assume in eqns.(6.1) and (6.2) that $\la_1+p^d+t \le \mu_1+p^d$, that is, $\la_1 +t \le \mu_1$. By Lemma 4.1(1) we have $\tbinom{\la_1+j_1}{j_1}= \tbinom{\la_1+p^d+j_1}{j_1}$ in $K$ since $p^d>\min\{\la_2, \mu_1-\la_1\} \ge j_1$. Thus the expressions in (6.1) and (6.2) are equal. Since $x^{+}_{1,t}$ generates the $S'$-module $M^+_1(t)$, we conclude the diagram commutes.

\textit{Case $i>1$.} This is similar to the previous case, but simpler. Consider $x_{i,t}$, where $i>1$ and $1 \le t \le \la_{i+1}$. Then the image of $x^+_{i,t}$ under the map $\psi_{i,t}\circ\pi_{i,t} \circ \Hom_S(\square_{\la},\Delta({\mu}))(\pi_{\D(\mu)} \circ\phi_T)$ is
\begin{align*}\nonumber&\;\\ \sum_{j_1+\dots+ j_{i+1}=t}\tbinom{a_{1i}+j_1}{j_1}\cdots\tbinom{a_{ii}+j_i}{j_i}&\begin{bmatrix*}[l]
		1^{(\la_1+p^d)}\cdots i^{(a_{1i}+j_1)}(i+1)^{(a_{1i+1}-j_1)} \cdots n^{(a_{1n})} \\
		2^{(a_{22})}\cdots i^{(a_{1i}+j_2)}(i+1)^{(a_{2i+1}-j_2)} \cdots n^{(a_{2n})}  \\
		3^{(a_{33})}\cdots i^{(a_{3i}+j_3)}(i+1)^{(a_{3i+1}-j_3)} \cdots n^{(a_{3n})}\\
		\cdots\\
		n^{(a_{nn})}  \end{bmatrix*}\end{align*}
and this is equal to image of $x^+_{i,t}$ under the map $\pi_{i,t}^+ \circ \Hom_{S'}(\square_{\la^+},\Delta({\mu^+}))\circ \psi(\pi_{\D(\mu)} \circ\phi_T)$. Hence the diagram commutes.
\subsection{Second example: $\lambda$ is a hook}
Suppose $\la \in \Lambda^+(n,r)$ is a hook, that is, a partition of the form $h=(a,1^b)$, where $b \ge 0$ and $a+b=r$. We will show in this case that the bound $p^d>b$ in Theorem 1.1(1) may be improved slightly to $p^d>i$. This is indeed an improvement because for $i>b$ both $\Ext^i_{S}(\D(h), \D(\mu))$ and $\Ext^i_{S'}(\D(h ^+), \D(\mu ^+))$ are zero since both $\D(h)$ and $\D(h^+)$ have projective resolutions over $S$ and $S'$ respectively of length $b$ \cite{Ma}. 

We use the the explicit finite projective resolution $P_{*}(a,b)$ of $ \triangle (h)$,
$0 \rightarrow \dots \rightarrow P_2(a,b) \xrightarrow{\theta_{2}(a,b)} P_1(a,b) \xrightarrow{\theta_{1}(a,b)} P_0(a,b)
$ of \cite{Ma} which we now recall. We have $P_i(a,b)=\sum D(\al_1,\dots,\al_{b+1-i}),$
where the sum ranges over all sequences $(\al_1,\dots,\al_{b+1-i})$ of positive integers of length $b+1-i$ such that $\al_1+\dots+\al_{b+1-i}=a+b$ and $a \le \al_1 \le a+i$. For further use we note that \begin{equation}\al_j \le i+1  \end{equation} for all $j >1$. The differential $\theta_{i}(a,b), i \ge 1$, is defined be sending $x_1 \otimes \dots \otimes x_{b+1-i} \in D(\al_1,\dots,\al_{b+1-i})$ to
\[ \sum_{j=1}^{b+1-i} (-1)^{j+1}x_1 \otimes \dots\otimes \triangle(x_j) \otimes \dots  \otimes x_{b+1-i} \in D(\al_1,\dots,u,v,\dots,\al_{b+1-i}),
\]
where  $\triangle(x_j)$ is the image of $x_j$ under the two-fold diagonalization $D(\al_j) \rightarrow \sum D(u,v)$, where the sum ranges of all positive integers $u,v$ such that $u+v=\al_j$ and $D(\al_1,\dots,u,v,\dots,\al_{b+1-i})$ is a summand of $P_{i-1}(a,b)$ with $u$ located in position $j$.

For $T \in \mathrm{SST}_{\al}(\mu)$ and $t \in \{1,\dots,m\}$, where $\al =(\al_1, \dots,\al_m)$, let $T(t) \in \mathrm{Tab}(\mu)$ be obtained from $T$ by replacing each occurrence of $j>t$ by $j-1$. Thus $\pi_{\D(\mu)} \circ\phi_{T(t)} \in \Hom_S(D(\al_1,\dots,\al_t+\al_{t+1},\dots,\al_m),\D(\mu))$ and by extending linearly we obtain for each degree $i$ a map of vector spaces 
$$\Phi_{a,t}: \Hom_S(P_{i-1}(a,b),\D(\mu)) \rightarrow \Hom_S(P_{i}(a,b),\D(\mu)), \pi_{\D(\mu)} \circ \phi_T \mapsto \pi_{\D(\mu)} \circ \phi_{T(t)}.$$ We have suppressed the dependence of $\Phi_{a,t}$ on $i$ and $\mu$ in order to avoid overwhelming notation. According to \cite{MS2}, Remark 2.2, we have the following description for the differential of $\Hom_S(P_*(a,b),\D(\mu))$, \[\Hom_S(\theta_{i-1}(a,b),\D(\mu)) = \sum_{t \ge 1}(-1)^{t-1}\Phi_{a,t},\]
where it understood that $\phi_{T(t)}=0$ if $t>m$. 

\begin{theorem}
	Let $\la=(a,1^b)$, $\mu=(\mu_1,\dots,\mu_n)$  be partitions of $r$ and let $d$ be an integer. If $p^d >i$, then \[\Ext^i_{S}(\D(h), \D(\mu)) \simeq\Ext^i_{S'}(\D(h^+), \D(\mu ^+)).\]
\end{theorem}
\begin{proof}In order to show that the complexes $\Hom_S(P_{*}(a,b),\D(\mu))$ and $\Hom_{S'}(P_*(a+p^d,b),\D(\mu^+))$ are isomorphic it suffices to show that the following diagram commutes
	\begin{center}
		\begin{tikzcd}
			\Hom_S(D(\al),\Delta(\mu))\arrow{d}{\psi} \arrow{rrr}{ \Hom_S(\Phi_{a,t},\Delta({\mu}))}
			&&& \Hom_S(D(\beta),\Delta(\mu)) \arrow{d}{\psi_{i,t}} \\ \arrow{rrr}{\Hom_{S'}(\Phi_{a+p^d,t},\Delta({\mu}))}
			\Hom_{S'}(D(\al^+),\Delta(\mu^+)) 
			&&& \Hom_{S'}(D(\beta^+),\Delta(\mu^+))
		\end{tikzcd}
	\end{center}
	where $\al=(\al_1,\dots,\al_t, \al_{t+1},\dots,\al_{b+2-i})$, $\beta=(\al_1,\dots,\al_t+\al_{t+1},\dots,\al_{b+2-i})$ and the vertical maps are isomorphisms from Lemma 3.3(2). Note that in the diagram, the module $\Hom_S(D(\al),\Delta(\mu))$ is a summand of $\Hom_S(P_{i-1}(a,b),\D(\mu))$ of degree $i-1$ and thus by eqn. (6.3) we have $a_j \le (i-1)+1=i $ for all $j>1$.
	
	Let $T \in \sta$, \[ T=\begin{matrix*}[l]
		1^{(a_{11})}2^{(a_{12})}3^{(a_{13})} \cdots n^{(a_{1n})} \\
		2^{(a_{22})}3^{(a_{23})} \cdots n^{(a_{2n})} \\
		\cdots\\
		n^{(a_{nn})}  \end{matrix*}. \]
	By computing we have that the image of $\pi_{\D(\mu)} \circ\phi_T \in \Hom_S(D(\al),\Delta(\mu))$ in the clockwise direction of the above diagram is $\prod_{j=t}^{n-1}\binom{a_{1j}+a_{1j+1}}{a_{1j+1}}\pi_{\D(\mu^+)} \circ\phi_{T(t)^+}$
	and in the counterclockwise direction is 
	\[ \begin{cases}\binom{a_{11}+p^d+a_{12}}{a_{12}}\prod_{j=2}^{n-1}\binom{a_{1j}+a_{1j+1}}{a_{1j+1}}\pi_{\D(\mu^+)} \circ\phi_{T(t)^+}, & t=1 \\ \prod_{j=t}^{n-1}\binom{a_{1j}+a_{1j+1}}{a_{1j+1}}\pi_{\D(\mu ^+)} \circ\phi_{T(t)^+}, &t>1.\end{cases}\]
	Using the hypothesis we have $p^d>i \ge \al_j =\al_{12}+\al_{21} \ge \al_{12}$ and thus by Lemma 4.1(1), $\binom{a_{11}+p^d+a_{12}}{a_{12}}=\binom{a_{11}+a_{12}}{a_{12}}$ in $K$. Hence the diagram commutes.
\end{proof}

\section{Statements and Declarations}The authors declare they have no conflict of interest.
\section{Acknowledgments}
The authors wish to thank the anonymous referee for useful comments.

\end{document}